\documentclass[12pt]{amsart}
\usepackage{amssymb,latexsym}
\usepackage{enumerate}
\makeatletter
\@namedef{subjclassname@2010}{
  \textup{2010} Mathematics Subject Classification}
\makeatother
\newtheorem{thm}{Theorem}[section]

\newtheorem{lem}[thm]{Lemma}

\newtheorem{prop}[thm]{Proposition}

\newtheorem{Thm}{Theorem}
\newtheorem{Cor}{Corollary}
\newtheorem{Pro}{Proposition}
\numberwithin{equation}{section}

\def\M{\mathcal M}
\newcommand{\pr}{\text{Prob}}
\newcommand{\ex}{\mathbb{E}}
\newcommand{\norm}[1]{\left|\hspace*{2.5pt} \!\!\left|
#1\right|\hspace*{2.5pt} \!\!\right|}

\begin{document}
\baselineskip=17pt
\title{Randomness of character sums modulo $m$}

\author{Youness Lamzouri}
\author{Alexandru Zaharescu}

\address{Department of Mathematics \\
University of Illinois at Urbana-Champaign \\
273 Altgeld Hall, MC-382 \\
1409 W. Green Street \\
Urbana, Illinois 61801, USA}
\email{lamzouri@math.uiuc.edu}
\address{Department of Mathematics \\
University of Illinois at Urbana-Champaign \\
273 Altgeld Hall, MC-382 \\
1409 W. Green Street \\
Urbana, Illinois 61801, USA}
\email{Zaharesu@math.uiuc.edu}

\date{}

\begin{abstract} Using a probabilistic model, based on random walks on the additive group $\mathbb{Z}/m\mathbb{Z}$, we prove that the values of certain real character sums are uniformly distributed in residue classes modulo $m$.
\end{abstract}



\thanks{The First author is supported by a postdoctoral fellowship from the Natural
Sciences and Engineering Research Council of Canada. Research of the second author is supported by the NSF grant DMS-0901621.}

\maketitle

\section{Introduction}

\noindent A central question in number theory is to gain an understanding of character sums
 $$S_{\chi}(x)=\sum_{n\leq x} \chi(n),$$
 where $\chi$ is a Dirichlet character modulo $q$. When $q=p$ is a prime number and $\chi_p= \left(\frac{\cdot}{p}\right)$ is the Legendre symbol modulo $p$, the character sums $S_p(x)=S_{\chi_p}(x)$ encode information on the distribution of quadratic residues and non-residues modulo $p$ (see for example Davenport and Erd\"os \cite{DaEr}, and Peralta \cite{Pe}). In particular, bounds for the order of magnitude of $S_{p}(x)$ lead to results on the size of the least quadratic non-residue modulo $p$ (see the work of Ankeny \cite{An};  Banks, Garaev, Heath-Brown and Shparlinski \cite{BGHS}; Burgess \cite{Bu}; Graham and Ringrose \cite{GrRi}; Lau and Wu \cite{LaWu}; Linnik \cite{Li}; and Montgomery \cite{Mo}).

Quadratic residues and non-residues appear to occur in a rather random pattern modulo $p$, which suggests that the values of $\chi_p(n)$ mimic a random variable that takes the values $1$ and $-1$ with equal probability $1/2$. This fact was recently exploited by  Granville and  Soundararajan \cite{GrSo} while investigating the distribution of the values of Dirichlet $L$-functions attached to quadratic characters at $s=1$.  Furthermore, a result of Davenport and Erd\"os \cite{DaEr} shows that short real character sums are indeed random in some sense. More specifically, they established that the values $S_{p}(n+H)-S_{p}(n)$ are distributed according to a Gaussian distribution of mean zero and variance $H$ as $H\to\infty$ in the range $\log H/\log p\to 0$ when $p\to\infty$.

In this paper, we investigate a new aspect of the \emph{randomness} of these character sums. To describe our results, we first need some notation. Let $F(X)$ be a square-free polynomial of degree $d_F\geq 1$ over the finite field $\mathbb{F}_p= \mathbb{Z}/p\mathbb{Z},$ and define
$$
  S_p(F,k):= \sum_{n\leq k} \chi_p(F(n)),
$$
for all positive integers $k\leq p$.  Moreover, let $\Phi_p(F;m,a)$ be the proportion of positive integers $k\leq p$ for which $S_p(F,k)\equiv a \bmod m$; that is
$$ \Phi_p(F;m,a)=\frac{1}{p}|\{k\leq p: S_p(F,k)\equiv a \text{ mod } m\}|.$$
Since the values $\chi_p(F(n))$ are expected to be randomly distributed, one might guess that $\Phi_p(F;m,a)\sim 1/m$ for all $a\bmod m$ as $p\to\infty$. We show that this is indeed the case in Corollary 1 below, uniformly for all $m$ in the range $m=o((\log p)^{1/4})$ as $p\to \infty$. Our strategy is to introduce a probabilistic model for the values $S_p(F,k)$ based on random walks. A simple random walk on $\mathbb{Z}$ is a stochastic process $\{S_k\}_{k\geq 1}$ where
$$ S_k= X_1+\cdots + X_k,$$
and $\{X_j\}_{j\geq 1}$ is a sequence of independent random variables taking the values $1$ and $-1$ with equal probability $1/2$ (for further reference see Spitzer \cite{Sp}). We shall model the values $S_p(F,k) \bmod m$ by the stochastic process $\{S_k\bmod m\}$ which may be regarded as a simple random walk on the additive group $\mathbb{Z}/{m\mathbb{Z}}$. To this end we consider the random variable
 $$\Phi_{\text{rand}}(N;m,a):=\frac{1}{N}|\{k\leq N: S_k\equiv a\bmod m\}|.$$
Here and throughout  ${\Bbb E}(Y)$ will denote the expectation of the random variable $Y$.
We first study the probabilistic model and prove
\begin{Pro} Let $m\geq 2$ be a positive integer. Then, for all $N\geq m^2$ we have
 $$ \sum_{a=0}^{m-1}\ex\left(\left(\Phi_{\textup{rand}}(N;m,a)-\frac{1}{m}\right)^2\right)
\ll \frac{m^2}{N}.$$
\end{Pro}
Appealing to Markov's inequality, we deduce from this result that $$\Phi_{\textup{rand}}(N;m,a)=\frac{1}{m}(1+o(1))$$ with probability $1-o(1)$ provided that $N/m^2\to\infty.$

Using Proposition 1, we establish an analogous estimate for the second moment of the difference $\Phi_p(F;m,a)-1/m$ (which may be regarded as the ``variance'' of $\Phi_p(F;m,a)$).
\begin{Thm} Let $p$ be a large prime number and $F(X)\in \mathbb{F}_p(X)$ be a square-free polynomial of degree $d_F\geq 1$. Then, for any integer $2\leq m\ll (\log p)^{1/4}$ we have
$$ \sum_{a=0}^{m-1}\left(\Phi_p(F;m,a)-\frac{1}{m}\right)^2
\ll_{d_F} \frac{m^2}{\log p}.
$$
\end{Thm}
As a consequence, we obtain
\begin{Cor} Under the same assumptions of Theorem 1, we have uniformly for all $0\leq a\leq m-1$
$$\Phi_p(F;m,a)= \frac{1}{m}+ O_{d_F}\left(\frac{m}{\sqrt{\log p}}\right).$$
\end{Cor}

Let $R_p(F,k)$ be the number of positive integers $n\leq k$ such that $F(n)$ is a quadratic residue modulo $p$, and similarly denote by $N_p(F,k)$ the number of $n\leq k$ for which $F(n)$ is a quadratic non-residue mod $p$.
Using a slight variation of our method we also prove that the values $R_p(F,k)$ (and $N_p(F,k)$) are uniformly distributed in residue classes modulo $m$. In this case, the corresponding probabilistic model involves random walks on the non-negative integers, where each step is $0$ or $1$ with equal probability. Define
$$ \widetilde{\Phi}_p(F;m,a)=\frac{1}{p}|\{k\leq p: R_p(F,k)\equiv a \text{ mod } m\}|.$$
Then, using a similar result to Proposition 1 in this case (see Proposition 3.3 below) we establish
\begin{Thm} Let $p$ be a large prime number and $F(X)\in \mathbb{F}_p(X)$ be a square-free polynomial of degree $d_F\geq 1$. Then, for any integer $2\leq m\ll (\log p)^{1/4}$ we have $$ \sum_{a=0}^{m-1}\left(\widetilde{\Phi}_p(F;m,a)-\frac{1}{m}\right)^2 \ll_{d_F} \frac{m^2}{\log p}.$$ A similar result holds replacing $R_p(F,k)$ with $N_p(F,k)$.
\end{Thm}

An important question in the theory of random walks on finite groups is to investigate how close is the distribution of the $k$-th step of the walk to the uniform distribution on the corresponding group (see for example Hildebrand \cite{Hi}). In our case this corresponds to investigating the distribution of $S_k\bmod m$. Define
$$ \Psi_{\text{rand}}(k;m,a)= \pr(S_k\equiv a \bmod m).$$

\begin{Pro} Let $m\geq 3$ be an odd integer and $0\leq a\leq m-1$. Then
$$
\Psi_{\textup{rand}}(k;m,a)= \frac{1}{m}+ O\left(\exp\left(-\frac{\pi^2k}{3m^2}\right)\right).
$$
\end{Pro}
This shows that the distribution of $S_k$ is close to the uniform distribution on $\mathbb{Z}/m\mathbb{Z}$ when $m=o(k^{1/2})$ as $k\to\infty$.
Although this result is classical (see for example Theorem 2 of Aldous and Diaconis \cite{AlDi}), we chose to include its proof for the sake of completeness.

We now describe an analogous result that we derive for character sums. Let $N$ be large, and for each prime $p\leq N$, we consider the walk on $\mathbb{Z}/m\mathbb{Z}$ whose $i$-th step corresponds to the value of $\chi_p(q_i)\bmod m$, where $q_i$ is the $i$-th prime number. One might guess that as $p$ varies over the primes below $N$, the distribution of the $k$-th step of this walk will be close to the uniform distribution in $\mathbb{Z}/m\mathbb{Z}$, as $N, k\to\infty$ if $m=o(k^{1/2})$. Define
$$ S_k(p)= \sum_{j\leq k}\chi_p(q_j),$$
and
$$ \Psi_N(k;m,a)=\frac{1}{\pi(N)}|\{ p\leq N: S_k(p)\equiv a \bmod m\}|.$$
Here and throughout
$\log_j$ will denote the $j$-th iterated logarithm, so that $\log_1n
=\log n$ and $\log_j n=\log (\log_{j-1}n) $ for each $j\geq 2$. We prove
\begin{Thm} Fix $A\geq 1$. Let $N$ be large, and $k\leq A(\log_2 N)/(\log_3 N)$ be a positive integer. Then we have
 $$\Psi_N(k;m,a)= \Psi_{\textup{rand}}(k;m,a) + O_A\left(\frac{1}{\log^AN}\right).$$
\end{Thm}
Hence, using Proposition 2 we deduce
\begin{Cor} Let $m$ be an odd integer such that $3\leq m\leq k^{1/2}$. Then under the same assumptions of Theorem 3 we have  uniformly for all $0\leq a\leq m-1$ that
 $$\Psi_N(k;m,a)= \frac{1}{m} + O_A\left(\exp\left(-\frac{\pi^2k}{3m^2}\right)+\frac{1}{\log^AN}\right).$$
\end{Cor}
We remark that under the assumption of the Generalized Riemann Hypothesis for Dirichlet $L$-functions, we can improve the range of validity of Theorem 3 to $k\ll (\log N)/(\log_2 N).$
\section{Preliminary lemmas}

 In this section we collect together some preliminary results which will be useful in our subsequent work. Here and throughout we shall use the notation $e_m(x)= \exp\left(\frac{2\pi ix}{m}\right).$ Recall the orthogonal relation
 \begin{equation}
 \frac{1}{m}\sum_{t=0}^{m-1}e_m(tn)=
\begin{cases} 1 & \text{ if } n\equiv 0\bmod m,\\ 0 & \text{ otherwise.} \end{cases}
\end{equation}
 Our first lemma gives the classical bound for incomplete
exponential sums over $\mathbb{F}_p$ of the form
\begin{equation*}
S_I(P_1,P_2)=\sum_{n \in I}
\chi_p(P_1(n))e_p(P_2(n)),
\end{equation*}
where $I$ is a subinterval of $\{0,1,\ldots,p-1\},$
and $P_1(X),\,P_2(X) \in
  \mathbb{F}_p[X],$ such that $P_1(X)$ is a nontrivial square-free polynomial.
\begin{lem}
Let $p\geq 3$ be a prime number and $I,\  P_1(X),\  P_2(X)$ be as above.
Then we have
$$
  |S_I(P_1,P_2)| \leq 2 D\sqrt{p}\log p,
$$
where
$$
D = \deg P_1(X) + \deg P_2(X).
$$
\end{lem}

\begin{proof}
First if $I = \{0,\ldots,p-1\},$ then $S_I(P_1,P_2)=S(P_1,P_2)$ is a complete sum and the result
follows from the classical Weil bound for exponential sums \cite{We}:
\begin{equation}
|S(P_1,P_2)|\leq Dp^{1/2}.
\end{equation}
Now, if $I$ is proper subinterval of $\{0,\ldots,p-1\},$ we shall use a standard procedure to
express our incomplete sum in terms of complete sums
of the same type. Using equation (2.1)
we see that
\begin{equation*}
S_I(P_1,P_2)=\sum_{n \bmod p}\chi_p(P_1(n))e_p(P_2(n))
\left( \sum _{m\in I}\frac1p
\sum_{t \bmod p}e_p(t(m-n)) \right).
\end{equation*}
Changing the order of summation and noting that the
inner double sum is a product of two sums, one  being
a geometric progression and the other a complete
exponential sum, we obtain
\begin{equation}
\begin{split}
S_I(P_1,P_2)=&\frac{1}{p} \sum_{t \bmod p}
\Big (\sum_{m\in I}e_p(tm) \Big )
\Big (\sum_{n \bmod p}\chi_p(P_1(n)) e_p(P_2(n) - tn)
\Big )\\
=&\frac{1}{p} \sum_{t \bmod p}
F_I (t) S\big(P_1,\widetilde{P_2}\big),
\end{split}
\end{equation}
where $\widetilde{P_2}(X)=P_2(X)-tX$ and $F_I (t) = \sum _{m \in I} e_p(tm).$ If $t\equiv 0\bmod p$ then $F_I (t) = |I|$. Otherwise
if $I=\{ M+1,\ldots,M+N\} $, say,
then
$$
F_I (t) = \frac {e_p\big (t(M+1)\big ) - e_p\big (
t(M+N+1)\big )}
{1-e_p(t)}.
$$
Here the numerator has absolute value at most 2, while the absolute value
of the denominator  is $2|\sin(t\pi/p)|.$  Hence
$$
|F_I(t)| \le \left|\sin\left(\frac{t\pi}{p}\right)\right|^{-1} \le
\Big ( 2\left|\left|\frac{t}{p}\right|\right| \Big ) ^{-1},
$$
where $||\cdot||$ stands for the distance to the nearest integer.
As a set of representatives modulo $ p $ we choose
$ \{-\frac {p-1}2, \cdots,\frac {p-1}2\},$
so that for $t\neq 0$ in this set we have
\begin{equation}
|F_I(t)|\le \frac {p}{2|t|}.
\end{equation}
Now, we insert (2.2) and  (2.4) in
(2.3) to obtain
$$
|S_I(P_1,P_2)|\le \frac{D}{p^{1/2}}\left(|I|+
\sum_{1\leq |t|\le \frac{p-1}{2}}\frac {p}{2|t|} \right)\leq 2D\sqrt{p}\log p.$$
This completes the proof of the lemma.
\end{proof}
The following lemma will be later used to prove that the product of distinct shifts of a square-free polynomial cannot be a square in $\mathbb{F}_p(X)$.
\begin{lem}
Let $r\ge 2,$ and $z_1,\ldots, z_r, $ be distinct
elements of $\mathbb{F}_p$. Moreover, let $ \M $ be a nonempty finite subset of the algebraic closure
$ \overline{\mathbb{F}}_p$ of $\mathbb{F}_p$
with
$4|\M|<p^\frac 1r .$ Then there exists a \,
$ j\in \{1,\ldots,r\} $
such that the translate  $\M+z_j $ is not contained in
$\cup_{i \ne j} ( \M + z_i).$
\end{lem}

\begin{proof}

Suppose that  $(z_1,\ldots,z_r,\M) $ provides a counterexample
to the statement  of the lemma. Then clearly for any
nonzero
$ t \in \mathbb{F}_p,\,$ $(tz_1,\cdots,tz_r,$ $t\M)$
is also a counterexample.

We now use Minkowski's theorem on lattice points in a symmetric
convex body to find a nonzero integer $t$ such that
$$
\begin{cases}
\;\;|t| &\le p-1 \\
\norm{\frac {tz_1}p} &\le (p-1)^{-\frac 1r}\\
&\vdots \\
\norm{\frac {tz_r}p} &\le (p-1)^{-\frac 1r} \end{cases}
$$
Another way to express this is that there
  are integers
\begin{equation}\label{eqnumar}
\begin{cases}
|y_j|&\le p(p-1)^{-\frac 1r}\\
y_j&\equiv tz_j\pmod p
\end{cases}
\end{equation}
for any  $ j \in \{ 1,\ldots,r\}. $  Thus  $ (y_1,\ldots,y_r,t\M) $ provides a
counterexample. Now let $j_0 $ be such that
\begin{equation*}
|y_{j_0}|=\max_{1 \le j \le r } |y_j|.
\end{equation*}
Choose $ \alpha \in  t\M $ and consider the set
$\tilde \M =t\M \cap (\alpha +\mathbb{F}_p) .$
Then  $ (y_1,\ldots,y_r, \tilde{\M }) $ will also be a counterexample.

Note that  $ \alpha +\mathbb{F}_p $ can be written as a union of
   $ |\M| $ intervals
whose endpoints are in $\tilde \M .$ Let $ \{ \alpha +a,\alpha +a+1,\cdots ,
\alpha +b \}$  be the longest of these intervals. Then
\begin{equation*}
|b-a| \ge  \frac {p}{| \tilde \M|}  \ge \frac {p}{|\M|}.
\end{equation*}
By this, \eqref{eqnumar} and  the hypothesis
$ 4|\M| < p^ {\frac 1r}$ we deduce

\begin{equation*}
|b-a| >4 p^{1-\frac 1r} >2 |y_{j_0}|.
\end{equation*}
Now the point is that if  $y_{j_0} >0 $ then  $\alpha +a+y_{j_0} $
belongs to $ \tilde \M +y_{j_0} $ but does not belong to
$ \bigcup _{i \ne j_0} ( \tilde{\M} +y_i ) ,$
while if  $y_{j_0} < 0 $ then  $ \alpha +b+ y_{j_0} $ belongs to
$ \tilde \M +y_{j_0} $ but does not belong to
$ \bigcup _{i \ne j_0} ( \tilde \M +y_i) .$
This completes the proof of the lemma.
\end{proof}
Using this lemma, we prove the following result

\begin{lem} Let $F(X)\in \mathbb{F}_p(X) $ be a square-free polynomial of degree $d_F\geq 1$. Let $b_1, \dots, b_L$ be distinct elements in $\mathbb{F}_p$ such that
$L<(\log p)/\log(4d_F)$. Then, for any $a\in \mathbb{F}_p$ the polynomial
$$ H(X)= \prod_{j=1}^L F(aX+b_j),$$
is not a square in $\mathbb{F}_p(X)$.
\end{lem}
\begin{proof}
Let $\alpha_1, \dots, \alpha_s$ be the roots of $F(X)$ in $\overline{\mathbb{F}}_p$. Since $F(X)$ is square-free
then the $\alpha_j$ are distinct and $s=d_F$. Let $\M= \{a^{-1}\alpha_1,\dots, a^{-1}\alpha_s\}$, and write
 $z_j=-a^{-1}b_j$ for all $1\leq j\leq L$. Then note that $\M +z_j$ is the set of the roots of $F(ax+b_j)$ in $\overline{\mathbb{F}}_p$. By our hypothesis it follows that $4|\M|< p^{1/L}$. Hence, we infer from Lemma 2.2 that there exists a $j\in \{1,\dots, L\}$ such that at least one of the roots of $F(ax+b_j)$ is distinct from all the roots of $\prod_{l\neq j}F(ax+b_l)$. This shows that $H(X)$ is not a square in $\mathbb{F}_p(X)$ as desired.

\end{proof}

\section{Random walks on the integers modulo $m$}

In this section we shall study the distribution of the random walk $\{S_k\bmod m\}_{k\geq 1}$ and prove Propositions 1 and 2. To this end, we establish the following preliminary lemmas.
\begin{lem} If $m\geq 3$ is an odd integer, then
\begin{equation}\max_{1\leq t\leq m-1}\left|\cos\left(\frac{2\pi t}{m}\right)\right|\leq 1-\frac{\pi^2}{3m^2},
\end{equation}
and
$$ \max_{1\leq t\leq m-1}\left|1+e_m(t)\right|\leq 2-\frac{\pi^2}{6m^2}.$$
\end{lem}

\begin{proof}
We begin by proving the first assertion. If $m\geq 5$ is odd, then
$$\max_{1\leq t\leq m-1}\left|\cos\left(\frac{2\pi t}{m}\right)\right|= \cos\left(\frac{2\pi}{m}\right).$$ Moreover we know that $\cos(x)\leq 1-x^2/3$ for $0\leq x\leq \pi/2$. This yields
$$ \max_{1\leq t\leq m-1}\left|\cos\left(\frac{2\pi t}{m}\right)\right|\leq 1-\frac{4\pi^2}{3m^2}.$$
Now, when $m=3$ we have $\max_{1\leq t\leq 2}|\cos(2\pi t/m)|= \cos(\pi/m)\leq 1-\pi^2/(3m^2)$. This establishes the first part of the lemma.

Moreover, we have
$$ |1+e_m(t)|^2= 2+2\cos(2\pi t/m)\leq 4\left(1-\frac{\pi^2}{6m^2}\right),$$
 which follows from (3.1). Therefore, using that $\sqrt{1-x}\leq 1-x/2$ for $0\leq x\leq 1$ we obtain the second assertion of the lemma.
\end{proof}
\begin{lem} If $m\geq 2$ is an integer, then
$$\sum_{t=1}^{m-1}\sum_{1\leq j_1<j_2\leq N} \cos\left(\frac{2\pi t}{m}\right)^{j_2-j_1}= O(m^3N),$$
and
$$ \sum_{t=1}^{m-1}\sum_{1\leq j_1<j_2\leq N} \left(\frac{1+e_m(t)}{2}\right)^{j_2-j_1}= O(m^3N).$$
\end{lem}
\begin{proof}
We prove only the first statement, since the proof of the second is similar. For $d\in\{1,\dots,N-1\}$, the number of pairs $1\leq j_1<j_2\leq N$ such that $j_2-j_1=d$ equals $N-d$. Therefore, the sum we are seeking to bound equals
\begin{equation} \sum_{t=1}^{m-1}\sum_{d=1}^{N-1}(N-d)\cos\left(\frac{2\pi t}{m}\right)^d.
\end{equation}
First, when $m$ is odd, Lemma 3.1 implies that the last sum is
$$\leq mN\sum_{d=1}^{N-1}\max_{1\leq t\leq m-1}\left|\cos\left(\frac{2\pi t}{m}\right)\right|^d\leq \frac{mN}{1-\max_{1\leq t\leq m-1}\left|\cos\left(\frac{2\pi t}{m}\right)\right|}\leq \frac{3m^3N}{\pi^2}.$$
Now, when $m=2r$ is even, then either $\cos(\pi t/r)=-1$ or $|\cos(\pi t/r)|<1$. In the latter case the proof of Lemma 3.1 implies that $|\cos(\pi t/r)|\leq 1-\pi^2/(3r^2).$ Hence, in this case we obtain
$$ \sum_{d=1}^{N-1}(N-d)\left|\cos\left(\frac{\pi t}{r}\right)\right|^d\ll m^2N.$$
On the other hand if $\cos(\pi t/r)=-1$, then our sum become
$$\sum_{d=1}^{N-1}(N-d)(-1)^d \leq 2 N.$$
This completes the proof.
\end{proof}
We begin by proving Proposition 2 first, since its proof is both short and simple.
\begin{proof}[Proof of Proposition 2]
Recall that
$$ \Psi_{\text{rand}}(k;m,a)= \pr(X_1+\cdots+X_k\equiv a\bmod m)=\frac{1}{2^k} \sum_{\substack{{\bf v}=(v_1,\dots,v_k)\in \{-1,1\}^k\\ v_1+\cdots+v_k\equiv a\bmod m}}1.$$
Hence, using (2.1)
we deduce
\begin{equation}
\Psi_{\text{rand}}(k;m,a)= \frac{1}{2^km}\sum_{{\bf v}=(v_1,\dots,v_k)\in \{-1,1\}^k}\sum_{t=0}^{m-1}e_m\Big(t\left(v_1+\cdots+v_k-a\right)\Big).
\end{equation}
The contribution of the term $t=0$ to the above sum equals $1/m$. Moreover, since $\sum_{\alpha\in\{-1,1\}}e_m(\alpha t)= 2\cos(2\pi t/m)$, then the contribution of the remaining terms equals
$$ \frac{1}{2^km}\sum_{t=1}^{m-1}e_m\left(-at\right)
\sum_{{\bf v}=(v_1,\dots,v_k)\in \{-1,1\}^k}e_m\Big(t(v_1+\cdots+v_k)\Big)= \frac{1}{m}\sum_{t=1}^{m-1}e_m\left(-at\right)\cos\left(\frac{2\pi t}{m}\right)^k.$$
Thus, the result follows upon using Lemma 3.1.
\end{proof}
\begin{proof}[Proof of Proposition 1] First, note that
$$ \Phi_{\text{rand}}(N;m,a)=\frac{1}{N}\sum_{j=1}^NY_j \ \ \text{ where } \ \ Y_j= \begin{cases} 1 & \text{ if } S_j\equiv a\bmod m \\ 0 & \text{ otherwise}. \end{cases}$$
On the other hand, if $ {\bf v}=(v_1,\dots,v_N)\in \{-1,1\}^N$, then (2.1) yields
$$|\{1\leq j\leq N: v_1+\dots+v_j\equiv a\bmod m\}|=
\frac{1}{m}\sum_{j=1}^N\sum_{t=0}^{m-1}e_m\Big(t(v_1+\cdots+v_j-a)\Big).$$
This implies
\begin{equation}
\begin{aligned}
&\ex\left(\left(\Phi_{\text{rand}}(N;m,a)-\frac{1}{m}\right)^2\right)= \frac{1}{2^N}
\sum_{{\bf v}=(v_1,\dots,v_N)\in \{-1,1\}^N}\left(\frac{1}{N}\sum_{\substack{1\leq j\leq N\\ v_1+\cdots+v_j\equiv a \bmod m}} 1-\frac{1}{m}\right)^2\\
&= \frac{1}{2^N(mN)^2}
\sum_{{\bf v}=(v_1,\dots,v_N)\in \{-1,1\}^N}\left|\sum_{j=1}^N\sum_{t=0}^{m-1}e_m
\Big(t(v_1+\cdots+v_j-a)\Big)-N\right|^2.\\
\end{aligned}
\end{equation}
Now, expanding the summand on the RHS of (3.4) we derive
\begin{equation*}
\begin{aligned}
&\left|\sum_{j=1}^N\sum_{t=0}^{m-1}e_m
\Big(t(v_1+\cdots+v_j-a)\Big)-N\right|^2= \left|\sum_{j=1}^N\sum_{t=1}^{m-1}e_m
\Big(t(v_1+\cdots+v_j-a)\Big)\right|^2\\
&= \sum_{1\leq t_1,t_2\leq m-1}e_m\big(a(t_2-t_1)\big)\sum_{1\leq j_1,j_2\leq N}e_m\Big(t_1(v_1+\cdots+v_{j_1})-t_2(v_1+\cdots +v_{j_2})\Big).
\end{aligned}
\end{equation*}
Hence, we infer from (2.1) that
\begin{equation}
\begin{aligned}
&\sum_{a=0}^{m-1}\left|\sum_{j=1}^N\sum_{t=0}^{m-1}e_m
\Big(t(v_1+\cdots+v_j-a)\Big)-N\right|^2\\
&= m\sum_{t=1}^{m-1}\sum_{1\leq j_1,j_2\leq N}e_m\Big(t\big((v_1+\cdots+v_{j_1})-(v_1+\cdots +v_{j_2})\big)\Big)\\
&=m^2N+m\sum_{t=1}^{m-1}\sum_{1\leq j_1<j_2\leq N}\Bigg(e_m\Big(t(v_{j_1+1}+\cdots+v_{j_2})\Big)
+e_m\Big(-t(v_{j_1+1}+\cdots+v_{j_2})\Big)\Bigg).
\end{aligned}
\end{equation}
Inserting this estimate into (3.4), and using that $\sum_{\alpha\in \{-1,1\}}e_m(\alpha t)=2\cos(2\pi t/m)$, we obtain
$$ \sum_{a=0}^{m-1}\ex\left(\left(\Phi_{\text{rand}}(N;m,a)-\frac{1}{m}\right)^2\right)
= \frac{1}{N}+ \frac{2}{mN^2}\sum_{t=1}^{m-1}\sum_{1\leq j_1<j_2\leq N} \cos\left(\frac{2\pi t}{m}\right)^{j_2-j_1}.$$
The result follows upon using Lemma 3.2 to bound the RHS of the last identity.
\end{proof}

In order to prove Theorem 2 we require an analogous result to Proposition 1 in the case of a random walk on the non-negative integers, where each step is $0$ or $1$ (rather than $-1$ or $1$). To this end, we take $\{\widetilde{X}_j\}_{j\geq 1}$ to be a sequence of independent random variables taking the values $0$ and $1$ with equal probability $1/2$, and define
$$ \widetilde{S}_k= \widetilde{X}_1+\cdots + \widetilde{X}_k,$$
and
$$\widetilde{\Phi}_{\text{rand}}(N;m,a)=\frac{1}{N}|\{1\leq j\leq N: \widetilde{S}_j\equiv a\bmod m\}|.$$
Using a similar approach to the proof of Proposition 1 we establish:
\begin{prop} Let $m\geq 2$ be a positive integer. Then, for all $N\geq m^2$ we have
 $$ \sum_{a=0}^{m-1}\ex\left(\left(\widetilde{\Phi}_{\textup{rand}}(N;m,a)-\frac{1}{m}\right)^2\right)
\ll \frac{m^2}{N}.$$
\end{prop}
\begin{proof} We follow closely the proof of Proposition 1. First, a similar analysis used to derive (3.4) allows us to obtain
\begin{equation}
\begin{aligned}
&\ex\left(\left(\widetilde{\Phi}_{\text{rand}}(N;m,a)-\frac{1}{m}\right)^2\right)\\
&= \frac{1}{2^N(mN)^2}
\sum_{{\bf v}=(v_1,\dots,v_N)\in \{0,1\}^N}\left|\sum_{j=1}^N\sum_{t=0}^{m-1}e_m
\Big(t(v_1+\cdots+v_j-a)\Big)-N\right|^2.\\
\end{aligned}
\end{equation}
Hence, using the identity (3.5) in equation (3.6) we get
\begin{equation}
\begin{aligned}
 &\sum_{a=0}^{m-1}\ex\left(\left(\widetilde{\Phi}_{\textup{rand}}(N;m,a)-\frac{1}{m}\right)^2\right)\\
&= \frac{1}{N}+\frac{1}{mN^2}\sum_{t=1}^{m-1}\sum_{1\leq j_1<j_2\leq N}\left(\left(\frac{1+e_m(t)}{2}\right)^{j_2-j_1}+\left(\frac{1+e_m(-t)}{2}\right)^{j_2-j_1}\right)\\
&= \frac{1}{N}+\frac{2}{mN^2}\sum_{t=1}^{m-1}\sum_{1\leq j_1<j_2\leq N}\left(\frac{1+e_m(t)}{2}\right)^{j_2-j_1},
\end{aligned}
\end{equation}
upon noting that
$$\sum_{t=1}^{m-1}\left(\frac{1+e_m(t)}{2}\right)^d=\sum_{r=1}^{m-1}\left(\frac{1+e_m(-r)}{2}\right)^d,$$
by making the simple change of variables $r=m-t$.
Appealing to Lemma 3.2 completes the proof.

\end{proof}


\section{Character sums with polynomials: proof of Theorems 1 and 2}

We begin by proving the following key proposition which establishes the required link with random walks. Let $p$ be a large prime number and $F(X)\in \mathbb{F}_p(X)$ be a square-free polynomial of degree $d_F\geq 1$ in $\mathbb{F}_p(X)$. Moreover, let $L\leq (\log p)/\log (4d_F)$ be a positive integer, and put $N=[p/L]-1.$  Furthermore, for any ${\bf v}=(v_1,\dots, v_L)\in \{-1,1\}^L$ we define
\begin{equation}
D_{p,F}({\bf v }, L)=\{0\leq s\leq N : \chi_p(F(sL+j))=v_j \text{ for all } 1\leq j\leq L\}.
\end{equation}
\begin{prop} Let $p$, $L$, and $F(X)$ be as above. Then for any ${\bf v}=(v_1,\dots, v_L)\in \{-1,1\}^L$ we have
$$|D_{p,F}({\bf v }, L)|= \frac{p}{2^LL}\Big(1+ O_{d_F}\left(p^{-1/10}\right)\Big).$$
\end{prop}

\begin{proof}
Let $S$ be the set of non-negative integers $0\leq s\leq N$ such that $F(sL+j)\neq 0$ for all $1\leq j\leq L.$ Then $|S|= N+O_{d_F}(1)$. Moreover, note that for $s\in S$ we have
\begin{equation} \frac{1}{2^L}\prod_{j=1}^L\left(1+v_j\chi_p(F(sL+j))\right)=\begin{cases} 1& \text{ if } s\in D_{p,F}({\bf v }, L),\\ 0 & \text{ otherwise}.  \end{cases}
\end{equation}
This yields
$$ |D_{p,F}({\bf v }, L)|= \frac{1}{2^L}\sum_{s=0}^N\prod_{j=1}^L\left(1+v_j\chi_p(F(sL+j))\right)+O_{d_F}(1).$$
Expanding the product on the RHS of the previous estimate, we find that $|D_{p,F}({\bf v }, L)|$ equals
\begin{equation}
\begin{aligned}
&\frac{1}{2^L}\sum_{s=0}^N\left(1+ \sum_{l=1}^L\sum_{1\leq i_1<i_2<\dots<i_l\leq L}v_{i_1}\cdots v_{i_l}\chi_p\big(F(sL+i_1)\cdots F(sL+i_l)\big)\right)+O_{d_F}(1).\\
&= \frac{N}{2^L} +  \frac{1}{2^L}\sum_{l=1}^L\sum_{1\leq i_1<\dots<i_l\leq L}v_{i_1}\cdots v_{i_l}\sum_{s=0}^N\chi_p\big(F(sL+i_1)\cdots F(sL+i_l)\big)+O_{d_F}(1).\\
\end{aligned}
\end{equation}
Since $F(X)$ is a square-free polynomial, then it follows from Lemma 2.3 that the polynomial $H_{i_1,\dots,i_l}(X)=F(LX+i_1)\cdots F(LX+i_l)$ is not a square in $\mathbb{F}_p(X)$.
Therefore, using Lemma 2.1 with $P_1(X)=H_{i_1,\dots,i_l}(X)$,  $P_2(X)=0$ and $I=\{0,\dots, N\}$, we obtain
$$ \left|\sum_{s=0}^N\chi_p\big(F(sL+i_1)\cdots F(sL+i_l)\big)\right|\leq 2d_FL\sqrt{p}\log p.$$
Inserting this bound in (4.3)  we get
\begin{equation}
|D_{p,F}({\bf v }, L)|= \frac{p}{2^LL}+ O_{d_F}\left(L\sqrt{p}\log p\right),
\end{equation}
which completes the proof.
\end{proof}

\begin{proof}[Proof of Theorem 1]
Recall that
$$ \Phi_p(F;m,a)=\frac{1}{p}|\{1\leq k\leq p: S_p(F,k)\equiv a \text{ mod } m\}|.$$
Let $L=[(\log p)/(\log(4d_F)],$ and put $N=[p/L]-1.$ Moreover, for any $0\leq s\leq N$, we define
$$M_L(s;m,a)= |\{1\leq l\leq L: S_p(F,sL+l)\equiv a \text{ mod } m\}|.$$
Then, note that
\begin{equation}
 \left|\Phi_p(F;m,a)-\frac{1}{m}\right|\leq \frac{1}{p}\sum_{s=0}^N\left|M_L(s;m,a)-\frac{L}{m}\right|+O\left(\frac{L}{p}\right).
\end{equation}
To bound the sum on the RHS of (4.5), we use the Cauchy-Schwarz inequality which gives
$$ \left(\sum_{s=0}^N\left|M_L(s;m,a)-\frac{L}{m}\right|\right)^2\leq (N+1)\sum_{s=0}^N \left(M_L(s;m,a)-\frac{L}{m}\right)^2.$$
Hence, combining this estimate with (4.5), we deduce
\begin{equation}
\left(\Phi_p(F;m,a)-\frac{1}{m}\right)^2\ll \frac{N}{p^2}\sum_{s=0}^N \left(M_L(s;m,a)-\frac{L}{m}\right)^2 +\frac{L^2}{p^2}.
\end{equation}
On the other hand, since $S_p(sL+l)=S_p(sL)+ \sum_{j=1}^l \chi_p\big(F(sL+j)\big)$, then
\begin{equation}
\sum_{a=0}^{m-1}\left(M_L(s;m,a)-\frac{L}{m}\right)^2=
\sum_{b=0}^{m-1}\left(\Delta_L(s;m,b)-\frac{L}{m}\right)^2,
\end{equation}
 where
  $$ \Delta_L(s;m,b)= |\{1\leq l\leq L: \sum_{j=1}^l \chi_p\big(F(sL+j)\big)\equiv b \text{ mod } m\}|.$$
Therefore, upon combining (4.6) and (4.7) we obtain
\begin{equation}
\sum_{a=0}^{m-1}\left(\Phi_p(F;m,a)-\frac{1}{m}\right)^2 \ll \frac{N}{p^2}\sum_{a=0}^{m-1}\sum_{s=0}^N \left(\Delta_L(s;m,a)-\frac{L}{m}\right)^2 +\frac{mL^2}{p^2}.
\end{equation}
Now we evaluate the inner sum on the RHS of the previous inequality. Using (2.1) we get
\begin{equation}
\begin{aligned}
 \sum_{s=0}^N \left(\Delta_L(s;m,a)-\frac{L}{m}\right)^2
 &= \frac{1}{m^2}\sum_{s=0}^N\left|\sum_{l=1}^L\sum_{t=0}^{m-1}e_m\Bigg(t\Big(\sum_{1\leq j\leq l}\chi_p\big(F(sL+j)\big)-a\Big)\Bigg)-L\right|^2\\
 &=\frac{1}{m^2}\sum_{s=0}^N\left|\sum_{l=1}^L\sum_{t=1}^{m-1}e_m\Bigg(t\Big(\sum_{1\leq j\leq l}\chi_p\big(F(sL+j)\big)-a\Big)\Bigg)\right|^2\\
&=\frac{1}{m^2}\sum_{{\bf v}\in \{-1,1\}^L} \left|\sum_{l=1}^L\sum_{t=1}^{m-1}e_m\Big(t\big(v_1+\cdots+v_l-a\big)\Big)\right|^2 D_{p,F}({\bf v},L).
\end{aligned}
\end{equation}
Hence, using Proposition 4.1 along with the identity (3.4) obtained in the random walk setting, we derive
\begin{equation*}
\begin{aligned}
&\sum_{s=0}^N \left(\Delta_L(s;m,a)-\frac{L}{m}\right)^2\\
&=\frac{p}{2^Lm^2L}\sum_{{\bf v}\in \{-1,1\}^L} \left|\sum_{l=1}^L\sum_{t=1}^{m-1}e_m\Big(t\big(v_1+\cdots+v_l-a\big)\Big)\right|^2 \left(1+O_{d_F}\left(p^{-1/10}\right)\right)\\
&=pL\ex\left(\left(\Phi_{\text{rand}}(L;m,a)-\frac{1}{m}\right)^2\right)
\left(1+O_{d_F}\left(p^{-1/10}\right)\right).\\
 \end{aligned}
\end{equation*}
Finally, combining this estimate with (4.8) we obtain
\begin{equation*}
\begin{aligned}
\sum_{a=0}^{m-1}\left(\Phi_p(F;m,a)-\frac{1}{m}\right)^2 &\ll_{d_F} \sum_{a=0}^{m-1}
\ex\left(\left(\Phi_{\text{rand}}(L;m,a)-\frac{1}{m}\right)^2\right)+\frac{m(\log p)^2}{p^2}\\
&\ll_{d_F} \frac{m^2}{\log p},
\end{aligned}
\end{equation*}
which follows from Proposition 1. This completes the proof.
\end{proof}
\begin{proof}[Proof of Theorem 2] We only prove the result for $R_p(F,k)$, since the proof for $N_p(F,k)$ is similar. Define
$$
\delta_F(j)= \begin{cases} 1 & \text{ if } \chi_p(F(j))=1\\ 0 & \text{ otherwise.} \end{cases}
$$
Then, note that
$$ R_p(F,k)=\sum_{j=1}^k \delta_F(j).$$
We follow closely the proof of Theorem 1. Let $L=[(\log p)/\log (4d_F)]$, and $N=[p/L]-1$. For any $0\leq s\leq N$ we define
  $$ \widetilde{\Delta}_L(s;m,b)= |\{1\leq l\leq L: \sum_{j=1}^l \delta_F(sL+j)\equiv b \bmod m\}|.$$
Then, similarly to the estimate (4.8) we obtain
\begin{equation}
\sum_{a=0}^{m-1}\left(\widetilde{\Phi}_p(F;m,a)-\frac{1}{m}\right)^2 \ll \frac{N}{p^2}\sum_{a=0}^{m-1}\sum_{s=0}^N \left(\widetilde{\Delta}_L(s;m,a)-\frac{L}{m}\right)^2 +\frac{m(\log p)^2}{p^2}.
\end{equation}
Moreover, an analogous approach which leads to the identity (4.9) also gives
$$ \sum_{s=0}^N \left(\widetilde{\Delta}_F(s;m,a)-\frac{L}{m}\right)^2= \frac{1}{m^2}\sum_{{\bf v}\in \{0,1\}^L} \left|\sum_{l=1}^L\sum_{t=1}^{m-1}e_m\Big(t\big(v_1+\cdots+v_l-a\big)\Big)\right|^2
\sum_{\substack{0\leq s\leq N\\ \delta_F(sL+j)=v_j \\ \text{ for all } 1\leq j\leq L}} 1.$$
Remark that if $F$ does not vanish in the interval $[sL+1,sL+L]$ then
$$\delta_{F}(sL+j)= \frac{1+\chi_p\big(F(sL+j)\big)}{2},$$
for all $1\leq j\leq L$. Hence, writing $\widetilde{\bf v}= (\widetilde{v}_1, \dots, \widetilde{v}_L)$ with $\widetilde{v}_j= 2v_j-1$, we deduce
$$\sum_{\substack{0\leq s\leq N\\ \delta_F(sL+j)=v_j\\ \text{ for all } 1\leq j\leq L}} 1= |D_p(\widetilde{\bf v }, L,F)|+O_{d_F}(1)=\frac{p}{2^LL}\Big(1+ O_{d_F}\left(p^{-1/10}\right)\Big) ,$$
which follows from Proposition 4.1.
Thus, appealing to the identity (3.6) obtained in the random walk setting, we derive
$$
\sum_{s=0}^N \left(\widetilde{\Delta}_F(s;m,a)-\frac{L}{m}\right)^2
= pL\ex\left(\left(\widetilde{\Phi}_{\text{rand}}(L;m,a)-\frac{1}{m}\right)^2\right)
\left(1+O_{d_F}\left(p^{-1/10}\right)\right).
$$
Therefore, inserting this estimate in (4.10) and using Proposition 3.3 we obtain
\begin{equation*}
\begin{aligned}
\sum_{a=0}^{m-1}\left(\widetilde{\Phi}_p(F;m,a)-\frac{1}{m}\right)^2 &\ll_{d_F} \sum_{a=0}^{m-1}\ex\left(\left(\widetilde{\Phi}_{\text{rand}}(L;m,a)-\frac{1}{m}\right)^2\right)
 +\frac{m(\log p)^2}{p^2}\\
&\ll_{d_F} \frac{m^2}{\log p},
\end{aligned}
\end{equation*}
as desired.
\end{proof}


\section{Character sums of fixed length: Proof of Theorem 3}

We shall derive Theorem 3 from the following proposition

\begin{prop} Fix $A\geq 1$. Let $N$ be large, and $k\leq A(\log_2N)/(\log_3N).$ Then for any ${\bf v}=(v_1,\dots,v_k)\in \{-1,1\}^k$ we have
$$ \frac{1}{\pi(N)}|\{p\leq N:  \chi_p(q_j)=v_j \textup{ for all } 1\leq j\leq k\}|= \frac{1}{2^k}\left(1+ O_A\left(\frac{1}{\log^AN}\right)\right).$$
\end{prop}
\begin{proof} If $\log N\leq p\leq N$ then
$$ \frac{1}{2^k}\prod_{j=1}^k\left(1+v_j\chi_p(q_j)\right)=
\begin{cases} 1 & \textup{ if } \chi_p(q_j)=v_j \textup{ for all } 1\leq j\leq k,\\
0 & \textup{ otherwise}.
\end{cases}
$$
Therefore we deduce that the number of primes $p\leq N$ such that   $\chi_p(q_j)=v_j$ for all $1\leq j\leq k$, equals
\begin{equation}
 \begin{aligned}
&=
\frac{1}{2^k}\sum_{p\leq N}\prod_{j=1}^k\left(1+v_j\chi_p(q_j)\right) + O(\log N)\\
&= \frac{1}{2^k} \sum_{p\leq N}\left(1+\sum_{l=1}^k\sum_{1\leq i_1<\cdots<i_l\leq k}v_{i_1}\cdots v_{i_l}\chi_p(q_{i_1}\cdots q_{i_{l}})\right)+ O(\log N)\\
&= \frac{\pi(N)}{2^k} +\frac{1}{2^k}\sum_{l=1}^k\sum_{1\leq i_1<\cdots<i_l\leq k}v_{i_1}\cdots v_{i_l}\sum_{ p\leq N}\left(\frac{q_{i_1}\cdots q_{i_l}}{p}\right)+ O(\log N).
\end{aligned}
\end{equation}
For $1\leq i_1<\cdots<i_l\leq k$ we let $Q_{i_1,\dots,i_l}= q_{i_1}\dots q_{i_l}$. Then it follows from the prime number theorem that
$Q_{i_1,\dots,i_l}\leq \prod_{j\leq k}q_j=e^{k\log k(1+o(1))}\leq (\log N)^{A+o(1)}.$
On the other hand, quadratic reciprocity implies that $\left(\frac{Q_{i_1,\dots,i_l}}{\cdot}\right)$ is a character of modulus $Q_{i_1,\dots,i_l}$ or $4Q_{i_1,\dots,i_l}$. Therefore, appealing to the Siegel-Walfisz Theorem (see Corollary 5.29 of Iwaniec-Kowalski \cite{IwKo}), we deduce
$$ \sum_{p\leq N} \left(\frac{Q_{i_1,\dots,i_l}}{p}\right)\ll_A (Q_{i_1,\dots,i_l})^{1/2}\frac{N}{\log^{2A} N}.$$
Inserting this estimate in (5.1) completes the proof.
\end{proof}

\begin{proof}[Proof of Theorem 3]
Using (2.1) we obtain
\begin{equation}
\begin{aligned}
 \Psi_N(k;m,a)&= \frac{1}{\pi(N)}|\{p\leq N: S_k(p)\equiv a\bmod m\}|.\\
 &= \frac{1}{m\pi(N)}\sum_{p\leq N}\sum_{t=0}^{m-1}e_m\big(t(S_k(p)-a)\big)\\
 &=  \frac{1}{m\pi(N)}\sum_{t=0}^{m-1}\sum_{{\bf v}\in \{-1,1\}^{k}}e_m\Big(t(v_1+\cdots+v_k-a)\Big)
 \sum_{\substack{p\leq N\\ \chi_p(q_j)=v_j \text{ for } 1\leq j\leq k}}1 \\
\end{aligned}
\end{equation}
Thus, appealing to Proposition 5.1 along with the identity (3.3) obtained in the random walk setting we derive
\begin{equation*}
\begin{aligned}
 \Psi_N(k;m,a)&=\frac{1}{2^km}\sum_{t=0}^{m-1}\sum_{{\bf v}\in \{-1,1\}^{k}}e_m\Big(t(v_1+\cdots+v_k-a)\Big)
+ O_A\left(\frac{1}{\log^AN}\right)\\
&= \Psi_{\text{rand}}(k;m,a)+ O_A\left(\frac{1}{\log^AN}\right),\\
\end{aligned}
\end{equation*}
which completes the proof.
\end{proof}


\end{document}